\documentclass[a4paper,11pt,leqno]{amsart}
\usepackage{etex}
\usepackage[a4paper,left=25mm,right=25mm,top=30mm,bottom=35mm,marginpar=25mm]{geometry}
\usepackage[utf8x]{inputenc}
\usepackage[english]{babel}
\usepackage{amsmath}
\usepackage{amssymb}
\usepackage{amsthm}
\usepackage{enumerate}
\usepackage{empheq}
\usepackage{fancybox}
\usepackage{xcolor}
\usepackage{mathrsfs}
\usepackage{slashed}
\usepackage{tikz}
\usepackage{harpoon}
\usepackage{setspace}
\usepackage{fancybox}
\usepackage{esint}
\usepackage{bm}
\usepackage{hyperref}
\usepackage{subfigure}

\newcommand{\R}{\mathbb{R}}

\newcommand{\N}{\mathbb{N}}

\newenvironment{dimo}{\textit{Proof.\ }}{\begin{flushright}$\square$\end{flushright}}

\newtheorem*{thmnonumber}{Theorem}

\newtheorem{thm}{Theorem}[section]
\newtheorem{defi}{Definition}[section]

\newtheorem{prop}[thm]{Proposition}
 
\newtheorem{cor}[thm]{Corollary}

\newtheorem{rmk}[thm]{Remark}

\title[Long time behaviour for the reinitialization of the distance function]{Long time behaviour for the reinitialization of the distance function}
\author{Marcello Carioni}
\address{
University of Graz,
Heinrichstraße 36,
8010, Graz,
Austria}
\email{marcello.carioni@uni-graz.at}
\begin{document}
\maketitle
\begin{abstract}
In this article we study the long-time behaviour of a class of non-coercive Hamilton-Jacobi equations, that includes, as a notable example, the so called reinitialization of the distance function. In particular we prove that its viscosity solution converges uniformly as $t\rightarrow +\infty$ to the signed distance function from the zero level set of the initial data.
\end{abstract}

\section{Introduction}
This article is concerned with the study of the long time behaviour for a class of evolutive Hamilton-Jacobi equations. In particular the model equation we consider is known as \emph{reinitialization of the distance function} and it has the following form: 
\begin{equation}\label{reinit}
\left\{
\begin{array}{ll}
u_t + f(u_0(x)) (|\nabla u| - 1) =0 & \mbox{in }\R^n \times (0,+\infty)\\
u(x,0) = u_0(x) & \mbox{in }\R^n\, ,
\end{array}
\right.
\end{equation}
where  $f(x) = \frac{x}{\sqrt{x^2 + \delta^2}}$ is a regularized version of the sign function for small $\delta > 0$. Additionally $u_0$ vanishes on a $n-1$ dimensional manifold denoted by $\Gamma$ and it is positive outside and negative inside (see Figure \ref{picture}).

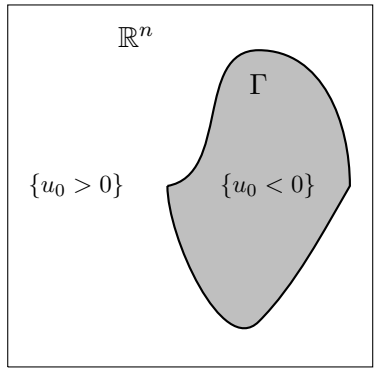
\begin{figure}[!h]
\centering
\begin{tikzpicture}[>=stealth, scale=0.6]
\draw (-3.5,-4) to (-3.5,4) to (4.5,4) to (4.5,-4) to (-3.5,-4); 
\draw  (0,0) [thick][black] [thick, fill=lightgray]
to[out=10,in=180, looseness=1] (2,3) node[black, scale=1] [below=5pt]{$\Gamma$}  node[black, scale=1] [above=5pt][left=35pt]{$\R^n$}    
to[out=0,in=90, looseness=1] (4,0)  node[black, scale=0.85] [left=11pt]{$\{u_0 < 0\}$} 
to[out=-120,in=45, looseness=0.8] (2,-3) 
to[out=225,in=-90, looseness=0.8] (0,0) node[black, scale=0.85] [left=15pt]{$\{u_0 > 0\}$}; 
\end{tikzpicture}
\caption{The zero level set of $u_0$}
\label{picture}
\end{figure}

This equation was introduced by Sussman, Smereka and Osher in \cite{ALS} in the context of the level set methods for incompressible two phase-flow and it is employed as a tool in all the numerical algorithms based on level set evolution (we refer to \cite{FPWC} for the fundational article on the level set method and to \cite{ACW},\cite{MLS}, \cite{DIBO}, \cite{LSMF} for examples of algorithms using level set methods and requiring the reinitialization of the distance function). 
The common feature of these algorithms is that they become numerically unstable when the gradient of the function that describes the evolving interface approaches zero. Therefore, in order to recover an efficient numerical scheme, it is customary to update it to be the sign distance function from the traced interface.
This is exactly the goal of Equation \eqref{reinit}.
To be more precise one can observe that, as $f(u_0(x)) = 0$ on $\Gamma$, the zero level set of the solution should be preserved during the evolution.
Hence, at least heuristically, the solution of \eqref{reinit} converges, as $t\rightarrow +\infty$, to the solution of the following eikonal equation (the steady state of \eqref{reinit}):
\begin{equation}\label{ei}
\left\{
\begin{array}{ll}
|\nabla \phi| =1 & \mbox{in }\R^n \setminus \Gamma\\
\phi(x) = 0 & \mbox{in }\Gamma\, ,
\end{array}
\right.
\end{equation}
which amdits as a solution the signed distance function from $\Gamma$.
The goal of this article is to formalize this heuristic observation in the framework of viscosity solutions (\cite{VSOH}, \cite{SPOV}) and in a more general setting. 

A related result about the reinitialization of the distance function is due to Hamamuki and Ntovoris in \cite{ARSF}. In this paper they build a framework to justify in a rigorous way the use of the reinizialitation procedure. In particular they couple a standard evolutive Hamilton-Jacobi equation with a variant of the reinitialization of the distance function and they consider the solution given by the alternate application of the two equations for time intervals time of lenght $k_1\varepsilon$ and $k_2\varepsilon$, respectively. Then they study the homogeneized equation for $\varepsilon \rightarrow 0$ and they recover the evolution of the distance function as $k_2/k_1 \rightarrow +\infty$. 
On the contrary, we focus on the long time behaviour of a suitable generalization of \eqref{reinit} using an approach that resembles more the work of Namah and Roqueouffre (\cite{ROT}) on coercive Hamilton-Jacobi equations in bounded domains. Additionally, the equation considered in \cite{ARSF} is substantially different (from a theoretical point of view) from Equation \eqref{reinit}: indeed the authors substitute the term $f(u_0(x))$ with $f(u(x))$ producing an Hamiltonian that is not depending explicitly on $x$ but on $u$.

We consider the following Hamilton-Jacobi equation that is a suitable generalization of \eqref{reinit}:
\begin{equation}\label{intr}
\left\{
\begin{array}{ll}
u_t + f(x) H(\|\nabla u\|)=0 & \mbox{in }\R^n \times (0,+\infty)\\
u(x,0) = u_0(x) & \mbox{in }\R^n\,,
\end{array}
\right.
\end{equation}
where $\|\cdot\|$ is an arbitrary norm in $\R^n$, $H(p) = 0$ if and only if $p=1$ and $f:\R^n \rightarrow \R$ is a bounded Lipschitz function that has the same sign as $u_0$. 
Denoting by $d_{\|\cdot\|_\star}(x,\Gamma)$ the signed distance function from $\Gamma$ with respect to the dual norm $\|\cdot\|_\star$ we prove the following result.

\begin{thmnonumber}
The viscosity solution of \eqref{intr} converges uniformly to $d_{\|\cdot\|_\star}(x,\Gamma)$ as $t\rightarrow \infty$ on every compact set of $\R^n$.
\end{thmnonumber}

We remark that apart from the structural hypothesis of $H$ that permit to recover the signed distance function in the limit (see hypothesis $(\textbf{H4})$ and $(\textbf{H5})$) and the classical assumptions at infinity (see hypothesis $(\textbf{H2})$ and $(\textbf{H3})$), we do not require addition regularity on $H$ (for example convexity).

The strategy to prove the stated result is based on the method of the half-relaxed limits for viscosity solutions. 
As in \cite{ROT}, we consider the following rescaled equation for $\varepsilon > 0$:
\begin{equation}\label{fast}
\left\{
\begin{array}{ll}
\epsilon u_t^\varepsilon + f(x) H(\|\nabla u^\varepsilon\|)=0 & \mbox{in } \R^n \times (0,+\infty)\\
u^\varepsilon(x,0) = u_0(x) & \mbox{in } \R^n\, ,
\end{array}
\right.
\end{equation}
where $u^\varepsilon (x,t) = u(x,t/\varepsilon)$ and then we prove that the half-relaxed limits of $u^\varepsilon (x,t)$ coincide with the signed distance function from $\Gamma$.
The striking difference to \cite{ROT} lies in the lack of coercivity of the Hamiltonian $f(x)H(\|\nabla u\|)$, due to the fact that $f(x) = 0$ on $\Gamma$. This does not allow to have a priori bounds of the gradient.
In order to overcome this difficulty we costruct barriers to the viscosity solution of \eqref{intr} that do not depend on $t$ in $\Gamma_\sigma$, a $\sigma$-neigheborhood of $\Gamma$ (see Proposition \ref{barrier}). More precisely we build two functions $v_\star$ and $v^\star$ that are viscosity subsolution (resp. supersolution) of \eqref{intr} positive outside $\Gamma$ and negative inside. 
In this way we manage to control the half relaxed limits of $u^\varepsilon (x,t)$ (and their zero level set) in $\Gamma_\sigma$. Outside $\Gamma_\sigma$ the Hamiltonian is coercive, thus we can apply suitable a priori estimate for the gradient (see Proposition \ref{unifo}).

Finally, in Theorem \ref{mainres} we prove the desired result, showing that the half-relaxed limits of $u^\varepsilon (x,t)$ coincide with the signed distance function. We employ a modification of the strong comparison principle for discontinuous viscosity solutions on unbounded domain, that allows to conclude that the signed distance function with respect to the dual norm $\|\cdot\|_\star$ is the only \emph{positive} discontinous viscosity solution of the eikonal equation (Corollary \ref{firstbound}).

\section{Setting and a priori estimates for the viscosity solution}
As anticipated in the introduction we consider the following Cauchy problem for an Hamilton-Jacobi equation:
\begin{equation}\label{eqn}
\left\{
\begin{array}{ll}
u_t + f(x) H(\|\nabla u\|)=0 & \mbox{in }\R^n \times (0,+\infty)\\
u(x,0) = u_0(x) & \mbox{in }\R^n\,,
\end{array}
\right.
\end{equation}
where $u_0 \in C^1(\R^n)$, $\|\cdot\|$ is an arbitrary norm in $\R^n$ and $n \geq 2$.
 
We denote by $\Gamma$ the zero level set of $u_0$:
\begin{equation*}
\Gamma:= \{x\in \R^n: u_0(x) = 0\}\, .
\end{equation*}
Moreover we call $D_+$ and $D_-$ the external and the internal part of $\Gamma$ (see Figure \ref{picture}):
\begin{equation*}
D_+:=\{x\in \R^n: u_0(x) > 0\} \quad D_-:= \{x\in \R^n: u_0(x) < 0\}\,.
\end{equation*}
Finally we denote by $\|\cdot\|_\star$ the dual norm of $\|\cdot\|$ and by $d_{\|\cdot\|}(x,\Gamma)$ the signed distance function from $\Gamma$ with respect to the norm $\|\cdot\|$, namely
\begin{displaymath}
d_{\|\cdot\|}(x,\Gamma) := 
\left\{\begin{array}{ll}
\inf\{\|x - y\| : y \in \Gamma\} & \mbox{if } x \in D_+\\
-\inf\{\|x - y\| : y \in \Gamma\} & \mbox{if } x \in D_-\, .
\end{array}
\right.
\end{displaymath}

\subsection{Assumptions on $f,u_0$ and $H$}
We assume the following hypothesis on $f(x)$ and on the initial data $u_0$:
\begin{itemize}
\item[(\textbf{G1})] $f$ is Lipschitz with Lipschitz constant $L>0$,
\vspace*{2mm}
\item[(\textbf{G2})] $\|f\|_\infty \leq C_1$,
\vspace*{2mm}
\item[(\textbf{G3})] $\Gamma=\{x\in \R^n: f(x) = 0\}, \quad D_+=\{x\in \R^n: f(x) > 0\}$, $\quad  D_-=\{x\in \R^n: f(x) < 0\}$,
\vspace*{2mm}
\item[(\textbf{G4})] $\inf_{x\in \Gamma} \|\nabla u_0(x)\| > 0$,
\vspace*{2mm}
\item[(\textbf{G5})] $\|\nabla u_0\|_\infty \leq C_2$. 
\end{itemize}
Moreover we make the following assumptions on the Hamiltonian $H$:
\begin{itemize}
\item [(\textbf{H1})] $H\in UC(B_R(0))$ for every $R>0$,
\vspace*{2mm}
\item [(\textbf{H2})] $\lim_{p\rightarrow +\infty} H(p) = + \infty$,
\vspace*{2mm}
\item [(\textbf{H3})] $|H(p)| \leq C_3 (1 + p)$,
\vspace*{2mm}
\item [(\textbf{H4})] There exists $\widetilde u \in C^1(\R^n)$ such that $0<\widetilde u < u_0$ in $D_+$, $u_0 < \widetilde u < 0$ in $D_-$ and $H(\|\nabla \widetilde u\|) \leq \alpha < 0$.
\vspace*{2mm}
\item [(\textbf{H5})] $H(p) = 0$ if and only if $p = 1$.
\end{itemize}
\begin{rmk}
The Hamilton-Jacobi equation for the reinitialization of the distance function \eqref{reinit} satisfies the previous hypothesis with 
\begin{equation*}
H(p) = p - 1 \quad \mbox{and} \quad f(x) =  \frac{u_0(x)}{\sqrt{u_0(x)^2 + \delta^2}}
\end{equation*}
for $\delta > 0$.
\end{rmk}

\begin{rmk}
The existence of $\widetilde u$ ensured by assumption $(\textbf{H4})$ will be employed in the construction of the barriers (see Proposition \ref{barrier}). We remark that it is not very restrictive: for example it holds for any Hamiltonian such that $H(p) < 0$ if $p<1$. Indeed, thanks to hypothesis $(\textbf{G5})$ it is enough to choose $\widetilde u = cu_0$ for an appropriate choice of the constant $c>0$.
\end{rmk}

\subsection{Existence of a unique solution and a priori estimates for the gradient}
The existence and uniqueness of the viscosity solution of \eqref{eqn} follow from the classical comparison principle for viscosity solutions of Hamilton-Jacobi equations and Perron's method developed by Ishii in \cite{PMFH}. 

For reader convenience we state the version of the comparison principle we use
(see for example \cite{OCA}, \cite{SDV}).
\begin{thm}\label{comparison}
Consider the following Hamilton-Jacobi equation
\begin{equation}\label{rr}
u_t + H(x,\nabla u) = 0 \quad (x,t) \in \R^n \times (0,T).
\end{equation}
Assume $H \in UC(\R^n \times B_R(0))$ for every $R>0$ and suppose that 
there exists a modulus of continuity $m:[0,+\infty) \rightarrow [0,+\infty)$ such that 
\begin{equation}
|H(x,p) - H(y,p)| \leq m(|x-y|(1+|p|)) \quad \forall \ x,y,p \in \R^n \, . 
\end{equation}
Let $u_1, u_2 \in UC([0,T]\times \R^n)$ be viscosity sub- and supersolution of \eqref{rr} respectively; then we have
\begin{displaymath}
\sup_{(x,t)\in  \R^n \times (0,T)} (u_1(x,t) - u_2(x,t)) \leq \sup_{x \in \R^n}(u_1(x,0) - u_2(x,0)).
\end{displaymath}
\end{thm}
As a consequence of hypothesis $(\textbf{G1})$, $(\textbf{H1})$ and $(\textbf{H3})$ the comparison principle holds for our class of Hamilton-Jacobi equations. 

\medskip

In order to prove existence of a viscosity solution for \eqref{eqn} it is enough to apply the Perron's method for Hamilton-Jacobi equations. In particular we have to exhibit a subsolution $u_*$ and a supersolution $u^*$ of \eqref{eqn} such that
\begin{equation}\label{boh}
u_*(x,0) \leq u_0(x) \leq u^*(x,0)
\end{equation}
for every $x\in \R^n$. 

One can readily verify, thanks to hypothesis (\textbf{G5}) and (\textbf{G2}), that there exists $C>0$ such that, $u^*(x,t) = u_0(x) + Ct$ and $u_*(x,t) = u_0(x) - Ct$ are a supersolution and a subsolution of \eqref{eqn} respectively and they satisfy \eqref{boh}.

\medskip

Hence there exists a unique viscosity $u \in UC((0,+
\infty) \times \R^n)$ for \eqref{eqn}.

\begin{prop}[A priori estimates]\label{unifo}
Given $u$ the viscosity solution of \eqref{eqn} there exists $C>0$ such that
\begin{equation}\label{onelip}
\|u_t\|_{L^{\infty}(\R^n \times \R_+)} \leq C 
\end{equation}
and for almost every $x \in \R^n \setminus \Gamma$ 
\begin{equation}\label{twolip}
\|\nabla u(x,\cdot)\|_{L^{\infty}(\R_+)} \leq \frac{C}{|f(x)|}\, .
\end{equation}
\end{prop}
\begin{proof}
Let $u(x,t)$ be a viscosity solution of $\eqref{eqn}$. For $h>0$ we have that also $u(x,t+h)$ is viscosity solution of \eqref{eqn}. Therefore by the comparison principle (Theorem \ref{comparison}) there holds 
\begin{equation}\label{linf}
\sup_{(x,t) \in \R^n \times \R_+}|u(x,t+h) - u(x,t)| \leq \sup_{x\in \R^n} |u(x,h) - u(x,0)|\, .
\end{equation}
Moreover thanks to hypothesis (\textbf{G5}) there exists $C>0$ such that, $u^*(x,t) = u_0(x) + Ct$ and $u_*(x,t) = u_0(x) - Ct$ are a supersolution and a subsolution of \eqref{eqn} respectively and they satisfy \eqref{boh}. Hence by the comparison principle we obtain
\begin{displaymath}
u_*(x,t) \leq u(x,t) \leq u^*(x,t) \quad \mbox{in } \R^n \times [0,+\infty)\,. 
\end{displaymath}
Thanks to \eqref{linf} we have
\begin{equation*}
\sup_{(x,t) \in \R^n \times \R_+} |u(x,t+h) - u(x,t)| \leq Ch
\end{equation*}
which implies that for every $x \in \R^n$, $u(x,t)$ is Lipschitz in $t$ and $\|u_t\|_{L^\infty(\R^n \times \R_+)} \leq C$.
In order to prove \eqref{twolip} we notice that by standard regularity results for Hamilton-Jacobi equations (see for example \cite{AITT}, Theorem 8.1), the viscosity solution $u(x,t)$ is locally Lipschitz in $x$. Hence, as a consequence of \eqref{onelip}, we infer that there exists $C>0$ such that
\begin{equation}\label{hypcontr}
H(\|\nabla u(x,t)\|) \leq \frac{C}{|f(x)|} 
\end{equation}
for almost every $(x,t) \in \R^n$ with $x\notin \Gamma$.

We will prove that there exists $C>0$ such that for every $\sigma > 0$ 
\begin{equation}\label{contr}
\|\nabla u(x,t)\| \leq \frac{C}{|f(x)|} \quad \mbox{a.e. in } (\R^n\setminus \Gamma_\sigma) \times \R_+\,.
\end{equation}
Then \eqref{twolip} follows from the arbitrariety of $\sigma$.
Suppose that \eqref{contr} does not hold. Then there exists $\sigma > 0$ such that for every $C>0$ we can find $A\subset (\R^n\setminus \Gamma_\sigma) \times \R_+$ of positive measure with 
\begin{equation*}
\|\nabla u(x,t)\| \geq \frac{C}{|f(x)|} \quad \mbox{for every } (x,t) \in A\, .
\end{equation*}
Hence, as $f(x)$ is bounded away from zero outside $\Gamma_\sigma$ (see assumption (\textbf{G3})) and thanks to hypothesis (\textbf{H2}), we have a contradiction with 
\eqref{hypcontr}.
\end{proof}

\begin{rmk}\label{nocoerc}
As estimate \eqref{twolip} shows, in contrast with \cite{ROT}, we cannot rely on uniform Lipschitz estimate up to $\Gamma$. This is a consequence of the lack of coercivity of the Hamiltonian close to $\Gamma$. 
\end{rmk}

\section{The eikonal equation}
This section is devoted to the study of the eikonal equation for a general norm $\|\cdot\|$ in unbounded domains. Given $\Omega \subset \R^n$ an open set (possibly unbounded) we consider the following Hamilton-Jacobi equation:
\begin{equation}\label{stat}
\left\{
\begin{array}{ll}
\|\nabla \phi\|=1 & \mbox{in }\Omega\\
\phi(x) = 0 & \mbox{in } \partial \Omega\,.
\end{array}
\right.
\end{equation}
Notice that even in the simple case of the euclidean distance we do not have uniqueness of solutions when $\Omega$ is unbounded. For example given $\Omega=(0,+\infty)$ it is clear that both $x$ and $-x$ are viscosity solutions of \eqref{stat}. However, at least for the euclidean norm, it is known that uniqueness hold when one is considering only positive viscosity solutions (see for example\cite{ARSF}, Lemma 3.5).

We want to generalize this result for a general norm $\|\cdot\|$ in the framework of discontinous viscosity solutions. 
Before that, we recall the definitions of discontinuous viscosity solutions and the known comparison principle when $\Omega$ is bounded (\cite{OCA},\cite{SDV}).
\begin{defi}[Discontinuous viscosity solutions]\label{viscousdisc}
Given $\Omega \subset \R^n$ open, consider the Hamilton-Jacobi equation:
\begin{equation}\label{hamilton2}
G(x,\nabla \phi)=0  \quad \mbox{in }\Omega\,,
\end{equation}
where $G:\Omega \times \R^n \rightarrow \R$ is a continuous Hamiltonian.

Let $u:\overline \Omega\rightarrow \R$ be locally bounded:
\begin{itemize}
\item[$\star$)]  If $u$ is upper semicontinous, we say that $u$ is a subsolution of \eqref{hamilton2} if for every $x_0 \in \overline\Omega$ and for every $v \in C^\infty(\overline\Omega)$ such that $u-v$ has a maximum in $x_0$ we have
\begin{displaymath}
G(x_0,\nabla v(x_0)) \leq 0.
\end{displaymath}
\item[$\star$)] If $u$ is lower semicontinous, we say that $u$ is a superolution of \eqref{hamilton2} if for every $x_0 \in \overline\Omega$ and for every $v \in C^\infty(\overline\Omega)$ such that $u-v$ has a minimum in $x_0$ we have
\begin{displaymath}
G(x_0,\nabla v(x_0)) \geq 0.
\end{displaymath}
\end{itemize}
A viscosity solution (discontinous) of \eqref{hamilton2} is a locally bounded function $u:\overline \Omega \rightarrow \R$ such that the upper semicontinuous envelope of $u$ is a subsolution of \eqref{hamilton2} and the lower semicontinuous envelope is a supersolution of \eqref{hamilton2}. \end{defi}

\begin{thm}[Comparison principle for discontinuous viscosity solutions]\label{compdisc}
Suppose that $\Omega \subset \R^n$ is an open, bounded set. 
Considering equation in \eqref{hamilton2}, suppose that $G(x,p) \in UC(\R^n \times B_R(0))$ for every $R>0$, it is convex in $p$ for every $x\in \Omega$ and there exists a function $\phi$ of class $C^1$ on $\Omega$ and continuous on $\overline \Omega$ such that
\begin{displaymath}
G(x,\nabla \phi) \leq \alpha < 0. \\
\end{displaymath}
Then given $u, v$ locally bounded in $\overline \Omega$ such that $u$ is a discontinuous subsolution of  \eqref{hamilton2} and $v$ is a discontinuous supersolution of \eqref{hamilton2}, with $u\leq v$ on $\partial \Omega$ we have
\begin{displaymath}
u(x) \leq v(x) \quad \mbox{ for every } x\in \Omega.
\end{displaymath}
\end{thm}

In order to prove the uniqueness in the class of positive solution for Equation \eqref{stat} we start with a Lipschitz estimate for the viscosity solution of $\|\nabla \phi\| = 1$. We employ a result due to Giga, Liu and Mitake \cite{SNPA} that we state here in our setting for the reader convenience: 
\begin{prop}[\cite{SNPA}]\label{giga}
Given a discontinuos viscosity solution of \eqref{stat}, for every $y \in \Omega$ and $r>0$ such that $B(x,r) \subset \Omega$, we have that
\begin{equation}
\|\phi(x) - \phi(y)\| \leq \|x-y\|_\star \quad \forall x \in B(y,r/4)\, .
\end{equation}
\end{prop}
Thanks to Proposition \ref{giga} we deduce a comparison principle for \emph{positive} viscosity solution of \eqref{stat}. 
\begin{cor}\label{firstbound}
Let $\phi$, locally bounded in $\Omega$ be a discontinuous viscosity solution of \eqref{stat}. Suppose in addition that $\phi(x)\geq 0$ for every $x\in \Omega$. Then $\phi(x)=d_{\|\cdot\|_\star}(x,\partial \Omega)$. 
\end{cor}
\begin{proof}
Consider $x \in \Omega$ and let $p(x) \in \partial \Omega$ be such that $\|x-p(x)\|_\star = d_{\|\cdot\|_\star}(x,\partial \Omega)$. Then, using Proposition \ref{giga} (see also Proposition 5.8 in \cite{SNPA} for a detailed argument) we have
\begin{equation*}
|\phi(x) - \phi(p(x))| \leq \|x-p(x)\|_\star\, .
\end{equation*}
Hence
\begin{equation}\label{read}
\phi(x) = |\phi(x) - \phi(p(x))| \leq \|x-p(x)\|_\star = d_{\|\cdot\|_\star}(x,\partial \Omega)\, .
\end{equation}
On the other hand, fixing an arbitrary $R>0$ and considering $\Omega_R := \Omega \cap B_R(0)$ we have that $\phi(x) \geq d_{\|\cdot\|_\star}(x,\partial \Omega_R)$, thanks to Theorem \ref{compdisc} and the positivity of $\phi$. 
Sending $R$ to $+\infty$ and using \eqref{read} we infer that $\phi(x) =  d_{\|\cdot\|_\star}(x,\partial \Omega)$.
\end{proof}

\section{Long time behaviour}

In order to overcome the difficulty pointed out in Remark \ref{nocoerc} we build barriers for the viscosity solution in such a way that we can control the zero level set of the solution uniformly in time close to $\Gamma$ (see \cite{ARSF} for an other example of application of this technique).

\subsection{Construction of the barriers}

\begin{prop}[Construction of the barriers]
\label{barrier}
There exists two locally Lipschitz functions $v^\star,v_\star : \R^n \times [0,+\infty) \rightarrow \R$ such that they are independent on $t$ in a $\sigma$-neighborhood of $\Gamma$ and
\begin{itemize}
\item [$i)$]$v_\star,v^\star > 0$ in $D_+ \times [0,+\infty)$,
\item [$ii)$]$v_\star,v^\star < 0$ in $D_- \times [0,+\infty)$, 
\item [$iii)$] $v_\star \leq u \leq v^\star$ in $\R^n \times [0,+\infty)$.
\end{itemize}
\end{prop}

\begin{dimo}
Notice firstly that thanks to (\textbf{G4}) there exists $M>0$ such that 
\begin{displaymath}
\inf_{x\in \Gamma} \|\nabla u_0(x)\| > 2M\,.
\end{displaymath}
Moreover as $u_0 \in C^1(\R^n)$ there exists $\sigma > 0$ such that $ \|\nabla u_0\|  \geq M$ in $\Gamma_{2\sigma}$.  

Define $\Gamma_\sigma^+ := \{x \in \R^2: 0\leq u_0(x) \leq \sigma\}$ and let us consider the following function $v^\star : \R^n\times [0,+\infty) \rightarrow \R$:
\begin{equation}\label{barr1}
v^\star(x,t) :=  \left\{
\begin{array}{ll}
k_1u_0(x)  & \mbox{in } \Gamma^+_\sigma \times [0,+\infty)   \\
k_1u_0(x)e^{k_2t(u_0 - \sigma)^2} & \mbox{in } (D_+ \setminus \Gamma^+_{\sigma}) \times [0,+\infty)\\
\widetilde u(x) & \mbox{in } D_- \times [0,+\infty)\, ,
\end{array}
\right.
\end{equation}
where $k_1,k_2>0$ will be choosen later. It is easy to verify that $v^\star$ is differentiable for every $x\notin \Gamma$. We want to prove that this defines a supersolution. 

For $x\in \Gamma^+_\sigma$ and $t\geq 0$ one has
$\partial_tv^\star(x,t) =  0$ and $\|\nabla v^\star\| = k_1\|\nabla u_0\| \geq  k_1M$. Hence choosing $k_1$ big enough, thanks to hypothesis (\textbf{H2}) we ensure that $v^\star$ is a supersolution. For every $x\in D_+ \setminus \Gamma^+_{\sigma}$ and $t\geq 0$ we have
\begin{displaymath}
\partial_t v^\star = 2k_1k_2u_0 (u_0 - \sigma)^2 e^{k_2t(u_0 - \sigma)^2} \mbox{ and } \|\nabla v^\star\| = k_1 \|\nabla u_0\| e^{k_2t(u_0 - \sigma)^2}| 1 + 2u_0(u_0 - \sigma) k_2 t|\,.
\end{displaymath}
Therefore given a point $x\in \Gamma^+_{2\sigma} \setminus \Gamma^+_\sigma$ we have that $\partial_t v^\star \geq 0$ and 
\begin{displaymath}
\|\nabla v^\star\| \geq k_1 M\,,
\end{displaymath}
so that, choosing $k_1$ big enough (again thanks to hypothesis (\textbf{H2})), $v^\star$ is a supersolution for $(x,t)\in (\Gamma^+_{2\sigma} \setminus \Gamma^+_\sigma) \times [0,+\infty)$. On the other hand given a point $x\in D_+ \setminus \Gamma^+_{2\sigma}$
\begin{displaymath}
\partial_t v^\star \geq k_1k_2\sigma^3\,. 
\end{displaymath}
So, as a consequence of (\textbf{G2}), it is enough to choose $k_2 \geq \frac{C_1 \min_{p\in \R^n} H(p)}{k_1\sigma^3}$ to infer that $v^\star$ is a supersolution in $D_+ \times [0,+\infty)$.\\
As for $x\in D_-$, using hypothesis (\textbf{H4}) we obtain
\begin{displaymath}
\partial_t v^\star + f(x)H(\|\nabla v^\star\|) = f(x)H(\|\nabla \widetilde  u\|) \geq 0, 
\end{displaymath} 
hence $v^\star$ is a subsolution in $D_-$.\\
Finally for $x\in \Gamma$ we have that $f(x)H(p) = 0$ for every $p\in \R$; therefore, as $v^\star$ does not depend on $t$ in a neighborhood of $\Gamma$, $v^\star$ is a supersolution in $\R^n \times [0, +\infty)$.\\
\vspace*{1mm}\\
Define $\Gamma_\sigma^- = \{x \in \R^2: -\sigma \leq u_0(x) \leq 0\}$ and as before consider the following function $v_\star : \R^n \times [0,+\infty) \rightarrow \R$:
\begin{equation}\label{barr2}
v_\star(x,t) =  \left\{
\begin{array}{ll}
\widetilde u (x)  & \mbox{in } D_+\times [0,+\infty) \\
k_1 u_0(x)  & \mbox{in } (D_- \cap \Gamma^-_\sigma) \times [0,+\infty)   \\
k_1 u_0(x)e^{k_2t(u_0 + \sigma)^2} & \mbox{in } (D_- \setminus \Gamma^-_\sigma) \times [0,+\infty)
\end{array}
\right.
\end{equation}
with $k_1$ and $k_2$ chosen as in the argument above. With similar computation to the first part of the proof it is easy to prove that $v_\star$ defines a subsolution.
\vspace*{2mm}\\
Properties $(i)$ and $(ii)$ of Proposition \ref{barrier} are clear. Moreover choosing $k_1$ big enough and using hypothesis (\textbf{H4}) we notice that $v_\star(x,0) \leq u_0(x) \leq v^\star(x,0)$. Hence by comparison principle we obtain
\begin{equation}
v_\star(x,t) \leq u(x,t) \leq v^\star(x,t)
\end{equation} 
for every $(x,t) \in \R^n \times [0,+\infty)$, that is $(iii)$. 

\end{dimo}

\begin{cor}\label{localbo}
Given $u(x,t)$ the viscosity of \eqref{eqn} we have that $u(x,t)$ is locally bounded in $\R^n$ uniformly in $t\in [0,+\infty)$. 
\end{cor}
\begin{proof}
From Proposition \ref{barrier} we have that
\begin{equation}
v_\star(x,t) \leq u(x,t) \leq v^\star(x,t)
\end{equation} 
for every $(x,t) \in \R^n \times [0,+\infty)$ where $v^\star$ and $v_\star$ are defined as in \eqref{barr1} and \eqref{barr2}. Therefore $u(x,t)$ is bounded in $\Gamma_\sigma$ for some $\sigma > 0$ and
\begin{equation}\label{bm}
|u(x,t)| \leq \widetilde C :=k_1 \max_{\partial \Gamma_\sigma} u_0(x) \qquad \mbox{in } \partial \Gamma_\sigma \times (0,+\infty)\,.
\end{equation}
Moreover by Proposition \ref{unifo} we have that 
\begin{equation}\label{lips}
\|\nabla u\|_{L^\infty(\R^2 \times \R_+)}\leq C \quad \mbox{ in } \R^2 \setminus \Gamma_\sigma\,.
\end{equation}
Suppose by contraddiction that there exists $x_0$ in $D_- \setminus \Gamma_\sigma$ such that $|u(x_0,t)| \rightarrow +\infty$ as $t\rightarrow +\infty$. If this is the case for every $n\in \N$ we can find $t_n > 0$ such that $|u(x_0,t_n)| \geq Cd_{\|\cdot\|_\star}(x_0, \Gamma) + \widetilde C + n$.
By the estimate \eqref{lips} we have also that there exists a segment $l_{x_0}$ passing through $x_0$ such that $\nabla u(x,t_n)$ exists and $\|\nabla u(x,t_n)\| \leq C$ for almost every $x \in l_{x_0} \cap (D_- \setminus \Gamma_\sigma)$. Hence, denoting by $x_1$ the first intersection of $l_{x_0}$ with $\partial \Gamma_\sigma$, we obtain using Equation \eqref{bm} that
\begin{eqnarray*}
Cd_{\|\cdot\|_\star}(x_0, \Gamma) + n & \leq & 
|u(x_0, t_n) - u(x_1,t_n)|\\
&=& \left|\int_{0}^1 \langle \nabla u(x_0 + h(x_1 - x_0), t_n), x_1-x_0\rangle \, dh \right|\\
&\leq & C\|x_1-x_0\|_\star\,, 
\end{eqnarray*}
where we use Cauhcy-Schwarz inequality. As the estimate holds for every $n\in \N$ we have a contradiction.
In an analogous way one reaches a contradiction supposing that there exists $x_0 \in D_+ \setminus \Gamma_\sigma$ such that $|u(x_0,t)| \rightarrow +\infty$ as $t\rightarrow +\infty$.

\end{proof}

\subsection{Convergence to the signed distance function}
In this subsection we prove that for $t\rightarrow +\infty$ the viscosity solution of \eqref{eqn} converges to the signed distance function uniformly on the compact sets of $\R^n$.

Given $\varepsilon > 0$ we consider the rescaling $u^\varepsilon (x,t) = u(x,t/\varepsilon)$. It is easy to verify that if $u$ is a viscosity solution of \eqref{eqn}, then $u^\varepsilon$ is a viscosity solution of 
\begin{equation}\label{homo}
\left\{
\begin{array}{ll}
\epsilon u_t^\varepsilon + f(x) H(\|\nabla u^\varepsilon\|)=0 & \mbox{in } \R^n \times (0,+\infty)\\
u^\varepsilon(x,0) = u_0(x) & \mbox{in } \R^n
\end{array}
\right.
\end{equation}
for every $\varepsilon > 0$. 

In order to pass to the limit as $\varepsilon \rightarrow 0$ we employ the method of half-relaxed limits to $u^\varepsilon$ and we refer to \cite{SDV} for a detailed presentation of this technique. 
We will denote by $\overline u$ and $\underline u$ the upper and the lower half-relaxed limit of $u^\varepsilon$ respectively, defined in the following way in a point $(x,t) \in \R^n \times [0,+\infty)$:
\begin{equation*}
\overline u(x, t) = \limsup_{(\tilde x,\tilde t, \varepsilon) \rightarrow (x,t,0)}u^\varepsilon(\tilde x,\tilde t) \quad \mbox{and} \quad \underline  u(x, t) = \liminf_{(\tilde x,\tilde t, \varepsilon) \rightarrow (x,t,0)}u^\varepsilon(\tilde x,\tilde t)\, .
\end{equation*}
Notice that thanks to Corollary \ref{localbo} the half-relaxed limit are well defined as $u^\varepsilon$ is locally bounded uniformly in $\varepsilon$.
 
We will use the following fundamental property of relaxed limits (see \cite{SDV} Lemma 4.1):
\begin{prop}[\cite{SDV}]\label{key}
If $K$ is a compact set of $\R^n$ and $\underline{u} = \overline u$ on $\R^n$, then 
\begin{equation*}
\lim_{\varepsilon \rightarrow 0} u^\varepsilon = \underline{u} = \overline u \quad \mbox{ uniformly on } K\, .
\end{equation*}
\end{prop}
We are now in position to prove the main result of the paper.
\begin{thm}\label{mainres}
Let $u(x,t)$ a be viscosity solution of \eqref{eqn}. Then $u(x,t)$ converges uniformly to $d_{\|\cdot\|_\star}(x,\Gamma)$ as $t\rightarrow \infty$ on every compact set of $\R^n$.
\end{thm}
\begin{dimo}
Thanks to Proposition \ref{barrier}, we have that for every $(x,t) \in \R^n \times [0,+\infty)$ and $\varepsilon>0$
\begin{displaymath}
v_\star\left(x,\frac{t}{\varepsilon}\right) \leq u^\varepsilon(x,t) \leq v^\star\left(x,\frac{t}{\varepsilon}\right)
\end{displaymath}
and, by construction, $v^\star$ and $v_\star$ depend on $x$ only in a $\sigma$-neighborhood of $\Gamma$. Therefore taking the upper and the lower half relaxed limits on both sides we obtain that $\underline u(x_0,t) = \overline u(x_0,t) = 0$ for every $x_0 \in \Gamma$.
Moreover as $v^\star(x,t)$ is strictly negative in $D_-\times [0,+\infty)$ and $v_\star(x,t)$ is strictly positive $D_+ \times [0,+\infty)$ we infer that the preservation of the zero level set is inherited by the approximate limits, namely
\begin{equation}\label{preservation}
\Gamma = \{x: \overline u(x,t) = 0\} = \{x: \underline u(x,t) = 0\}\,.
\end{equation}
for every $t\in [0,+\infty)$. 

By standard results on the stability of approximate limits (see \cite{SDV}, Section 4.3) we have that $\overline u$ (resp. $\underline u$) is a subsolution (resp. supersolution) of 
\begin{equation*}
f(x)H(\|\nabla u \|) = 0\, . 
\end{equation*}
This yields that $\overline u$ (resp. $\underline u$) is a subsolution (resp. supersolution) of the following equation in $D_+$:
\begin{equation}
H(\|\nabla u(x,t)\|) = 0\, .
\end{equation}
By hypothesis $(\textbf{H5})$ this holds if and only if $\overline u$ (resp. $\underline u$) is a subsolution (resp. supersolution) of
\begin{equation*}
\|\nabla u(x,t)\| = 1\,. 
\end{equation*}
Fixing $t_0 \in (0,+\infty)$ we have that $\overline u(x,t_0)$ is a subsolution and $\underline u(x, t_0)$ is a supersolution of the eikonal equation \eqref{stat} in $D^+$.
This implies by Corollary \ref{firstbound} and Equation \eqref{preservation} that for every $t_0 > 0$ we have
\begin{displaymath}
\overline u(x,t_0) \leq d_{\|\cdot\|_\star}(\Gamma,x) \leq \underline u(x,t_0) \quad \mbox{ in } D_+\, .
\end{displaymath}
Hence, for every $(x,t) \in D_+ \times (0,+\infty)$, we infer $\overline u(x,t) = \underline u(x,t) = d_{\|\cdot\|_\star}(x,\Gamma)$. \\
\vspace*{1mm}\\
On the other hand using again hypothesis $(\textbf{H5})$, we obtain that $\overline u$ (resp. $\underline u$) is a subsolution (resp. supersolution) of the following equation in $D_-$
\begin{displaymath}
-\|\nabla u(x,t)\| +1 =  0\, .
\end{displaymath} 
Fixing again $t_0 \in (0,+\infty)$, then $\overline u(x,t_0)$ is a subsolution and $\underline u(x,t_0)$ is a supersolution in $D_-$ of
\begin{equation*}
-\|\nabla u(x)\| + 1 = 0
\end{equation*}
and consequently $ - \overline u(x,t_0)$ is a supersolution and $- \underline u(x,t_0)$ is a subsolution in $D_-$ of
\begin{equation}\label{eik}
\|\nabla u(x)\| - 1 = 0\, .
\end{equation}
This implies by Corollary \ref{firstbound} and Equation \eqref{preservation} that for every $t_0 > 0$ we have
\begin{displaymath}
- \underline u(x,t_0) \leq \inf\{\|x-y\|_\star: y \in \Gamma\} \leq -\overline u(x,t_0) \quad \mbox{ in } D_-
\end{displaymath}
and hence
\begin{displaymath}
\overline u(x,t_0) \leq d_{\|\cdot\|_\star}(\Gamma,x) \leq \underline u(x,t_0) \quad \mbox{ in } D_- \, .
\end{displaymath}
So for every $(x,t) \in D_- \times (0,\infty)$, we have $\overline u(x,t) = \underline u(x,t) =d_{\|\cdot\|_\star}(\Gamma,x)$. 

Therefore
\begin{equation}
\underline u (x,t) = \overline u (x,t) = d_{\|\cdot\|_\star}(x,\Gamma) \quad \mbox{in } \R^n \times (0+\infty)\, .
\end{equation}
By Proposition \ref{key} we infer that 
\begin{displaymath}
u^\varepsilon(x,t) \rightarrow d_{\|\cdot\|_\star}(x,\Gamma) 
\end{displaymath}
as $\varepsilon \rightarrow 0$, uniformly on every compact set of $\R^n$.

Finally, recalling the definition of $u^\varepsilon$, one concludes.
\end{dimo}

\bibliographystyle{abbrv}
\bibliography{biblio}

\end{document}